\documentclass[a4paper,10pt,oneside]{article}
\pdfoutput=1
\usepackage{amsmath}
\usepackage{amssymb}
\usepackage{calrsfs}
\usepackage{amstext}
\usepackage{epsfig}
\usepackage{color}
\usepackage[english]{babel}

\usepackage{esvect}
\usepackage{amsthm}
\usepackage{amscd}   
\usepackage{esvect}

\newcommand{\BBN}{{\mathbb N}}

\newcommand{\BBC}{{\mathbb C}}

\newcommand{\cl}[1]{\begin{center}{#1}\end{center}}
\newcommand{\col}[2]{ \textcolor{#1}{#2}}

\newcommand{\dentro}[1]{$\mbox{Spivey\;\em \cite{spi}}.$}

  \newcommand{\dsum}[3]{\displaystyle \sum_{#1}^{#2}{#3}}

\newcounter {eser}[section]

  \newtheorem{example}{Example}[section]
   \newtheorem{remark}{Remark}[section]
   
  \newtheorem{lemma}{Lemma}[section]
  \newtheorem{proposition}{Proposition}[section]
  
  \newtheorem{theorem}{Theorem}[section]
\newcommand{\st}[2]{\mbox{$\left\{\begin{array}{c}{#1}\\{#2} \end{array}\right\}$}}

\newcommand{\srel}[2]{\stackrel{#1}{#2}}

\theoremstyle{definition}
\newtheorem{definition}[theorem]{Definition}

\newtheorem{identity2}{Identity}[section]
\theoremstyle{remark}

\newtheorem{observation}[theorem]{Observation}
\newcommand{\idi}[1]{\begin{identity2}{#1}\end{identity2}}

\begin{document}

\title{Generalized Harmonic Numbers: Identities and Properties}

 \author{Roberto S\'anchez-Peregrino Department of mathematics \\  Padue Università,\thanks{retiree:Universit\`a degli Studi di Padova, Dipartimento di Matematica Pura ed Applicata, Via Belzoni 7, I-35131 Padova (Italy).}\\ peregrino84@gmail.com   }
  \date{}

\maketitle
\begin{abstract}
This paper builds on the research initiated by Boyadzhiev, but introduces generalized harmonic numbers,
\[
H_n(\alpha)= \sum_{k=1}^n \frac{\alpha^{k}}{k},
\]
which enable the derivation of new identities as well as the reformulation of existing ones. We also generalize Gould's identity, allowing classical harmonic numbers to be replaced by their generalized counterparts. Our results contribute to a deeper understanding of the structural properties of these numbers and highlight the effectiveness of elementary techniques in uncovering new mathematical phenomena. In particular, we recover several known identities for generalized harmonic numbers and establish new ones, including identities involving generalized harmonic numbers together with Fibonacci numbers, Laguerre polynomials, and related sequences.
\end{abstract}

 \cl{\section{INTRODUCTION AND BACKGROUND}}
 In the first section, we establish several preliminary results and definitions that will be used throughout the paper.
In the second section, we present a new formulation and interpretation of Boyadzhiev's theorem:(\cite{khristo})
$\sum_{k=1}^n \binom{n}{k} \frac{a_k}{k+\lambda}.$ In particular, in this section, we generalize a Gould identity that enables us to establish a new relationship between harmonic and generalized harmonic numbers.
In the third section, we generalize a theorem due to Pan,(\cite{pan})
$\sum_{k=0}^n a^k b^{n-k} \binom{n}{k} H_k(\alpha),$ by employing the generalized harmonic numbers.
In the fourth section, we apply Sánchez's method along with Pan's theorem to address Boyadzhiev's question(\cite{khristo3}) .\par

\vspace{25pt}\par
MSC{ Mathematics Subject Classification 2010: Primary 11B65. Second 05A19?,40889? : \par
 \vspace{25pt}
 Keywords: Binomial coefficient, binomial transformation, stirling number of the second kind,
harmonic number, generalized harmonic number.}

\newpage
{\section{Definition \bf{ and Properties}, and so on.}}
\begin{description}

\item \begin{definition} Given a sequence $\{a_k\}_{k=0}^n$, we define its binomial transform as the new sequence  $\{b_k\}_{k=0}^n$  where
$$ b_n= \sum_{k=0}^n\binom{n}{k}a_k  $$.
\end{definition}

\begin{definition} Let $P$ =\{Normal power series\}which is a Integral Domain.
\end{definition}

\begin{definition}
$$
\binom{x}{y} = 
\begin{cases}
\frac{\Gamma( x+1)}{\Gamma(y+1)\Gamma(x-y+1)} & \text{if  }  x,y\in \BBC, \\
\frac{(x)_y}{y!} & \text{if }  x\in\BBC\;y\in \BBN,
\end{cases}
$$
\end{definition}
\begin{proposition}\label{sum}
Let $x\in \BBC$, and $n\in \BBN$ then, \par
$\sum_{m=0}^n\binom{x+m}{m}=\binom{x+n+1}{n}$
\end{proposition}
 \begin{definition}
 Let $\alpha \in \BBC$ and $p \in \BBN,$ we define
\item $ H_k^{(p)}(\alpha)=\begin{cases}0&{\text{if }k=0},\\
  \sum_{j=1}^k\frac{\alpha^j}{j^p}&{\text{otherwise.}}\end{cases}$\par
 \end{definition}
 
 \begin{remark} 
 
 $ H_k^{(p)}(\alpha)=\begin{cases}H_k(\alpha)\text{ (The generalized harmonic numbers)}&{\text{if }p=1},\\
 H_k\text{ (The  harmonic numbers)}&{\text{if }p=1,\alpha=1},\\
 - H_k^-\text{ (The skew-harmonic numbers)}&{\text{if }p=1,\alpha=-1}.
 \end{cases}$\par
 \end{remark}

 $ H_n(\alpha)$, these numbers have the same form as the harmonic numbers. They can also be expressed as:
$$ H_n(\alpha)=\int_{0}^{\alpha}\frac{1-y^n}{1-y}\;dy,$$
 $$H_n(\alpha)=\int_{1-\alpha}^{1}\frac{1-(1-x)^n}{x}\;dx.\;\text{Gencev} (\cite{Gencev}) $$
 %Os course $H_(1)=H_n.$
\begin{definition}
\label{stirling}
The numbers represented by the symbols \st{t}{s} are called Stirling numbers of the second kind.

\end{definition}
\item \begin{observation}[Sanchez\cite{sanchez}]\label{sanchezlemma}\par
 $$\binom{n}{k}k^p=\sum_{l=0}^p\sum_{j=0}^p(-1)^l\binom{n-l}{k}\binom{n}{l}\binom{n-l}{j-l}j!\st{p}{j}=
 {\sum_{l=0}^n\sum_{j=0}^n(-1)^l\binom{n-l}{k}\binom{n}{l}\binom{n-l}{j-l}j!\st{p}{j}}$$
 \begin{example}\label{exsanchez}
 $$\binom{n}{k}k^p =\left\{\begin{array}{l} n\binom{n}{k}-n\binom{n-1}{k},\quad \text{if} \; p=1;\\
  n^2\binom{n}{k}-n(2n-1)\binom{n-1}{k} +n(n-1)\binom{n-2}{k},\quad \text{if} \; p=2;\\
n^3\binom{n}{k}-n(3n^2-3n+1)\binom{n-1}{k}+3n(n-1)^2
\binom{n-2}{k}-\\ 
n(n-1)(n-2)\binom{n-3}{k},\; \text{if}\; p =3.\\
n^p\binom{n}{k}-n(n^p-(n-1)^p)\binom{n-1}{k}+\cdots+ (-1)^p\binom{n}{p}p!\binom{n-p}{k},\,\\ \text{for any }\quad p, n ;\\
 \end{array}\label{sanchez2}
 \right.$$
 \end{example}
\end{observation}
 \item \begin{observation}[Sanchez\cite{sanchez}]\label{sanchez}\par
Given the relation $\sum_{k=0}^n\binom{n}{k}a_k=b_n$. We compute the binomial transform of the sequence $\{k^pa_k\}$ in terms of 
 $\{b_n\}$ i.e.,\par
 $\sum_{k=0}^n\binom{n}{k}k^pa_k=\sum_{k=0}^n\sum_{l=0}^p\sum_{j=0}^p(-1)^l\binom{n-1}{k}\binom{n}{l}
 \binom{n-1}{j-1}j!\st{p}{j}a_k$ =\par
 $\sum_{l=0}^p\sum_{j=0}^p(-1)^l\binom{n}{l}\binom{n-l}{j-l}j!\st{p}{j}b_{n-l}.$\par

 \end{observation}
 \item 
 
\begin{remark}\label{remarkboyad}
Let \( f(t) \) be a function analytic on the unit disk:
\begin{equation}
f(t) = \sum_{n=0}^{\infty} a_n t^n. \label{gould1}
\end{equation}
Then the following identity holds (a proof can be found in \cite{khristo3}):
\begin{equation}
\frac{1}{1 - \lambda t} f\left( \frac{\mu t}{1 - \lambda t} \right) = \sum_{n=0}^{\infty} t^n \left( \sum_{k=0}^n \binom{n}{k} \mu^k \lambda^{n-k} a_k \right). \label{gouldboya1}
\end{equation}
Here, \( \lambda \) and \( \mu \) are suitable parameters.\par 

{We provide a new and completely different proof of this identity in Section~4 This paper is about the power series ring..}
\end{remark}
\item 
\begin{definition}[Iverson Bracket]
 Let P be a mathematical statement, then the Iverson bracket is defined by
$$[P]=\begin{cases}
0& \text{if  P is false}; \\ 
1& \text{if  P is true}.
\end{cases}$$
\end{definition}
\end{description}
The following section presents a reformulation of Boyadzhiev's work, leading to a generalization of Gould’s fundamental identity and establishing a transition from identities involving harmonic numbers to those involving generalized harmonic numbers.
{\section{Boyadzhiev.}}
In this section, we present a new formulation of Boyadzhiev's Theorem~\cite{khristo}.
\begin{lemma} \label{lemma}
  let $\{b_n\}$  be  a sequence, for every number $\lambda\in\BBC-[-1,-2,\dots,-n]$, and  $n\in \BBN^+$, we have 
\begin{equation}
n!\sum_{m=1}^n\frac{b_m}{m!(\lambda +m)(\lambda +m+1)...(\lambda +n)}=\begin{cases}\sum_{m=1}^n\frac{b_m}{m},&{\text{if }\lambda=0},\\
\frac{1}{\lambda \binom{n+\lambda}{n}}[\sum_{m=1}^{n}\binom{\lambda-1+m}{m}b_m],&{\text{if }\lambda\ne 0},\\
H_n,&{\text{if }\lambda=0,b_m=1},\\
\frac{1}{\lambda \binom{n+\lambda}{n}}[\binom{\lambda+n}{n-1}-1],&{\text{if } \lambda \ne 0,b_m=1.}
\label{lemmaeq0}
\end{cases}
\end{equation}
\end{lemma}
 \begin{proof} If $\lambda\ne 0$, the LHS   (\ref{lemmaeq0}) is  
\begin{eqnarray*}
  \lefteqn{n!\sum_{m=1}^n\frac{b_m}{m!(\lambda +m)(\lambda +m+1)...(\lambda +n)}=}\\
    &&\frac{n!}{(n+\lambda)_n}\frac{1}{\lambda}\left(\sum_{m=1}^{n}\frac{\lambda{(\lambda+1)(\lambda+2)\cdots (\lambda+m-1)}b_m}{m!}\right)=\\
     &&\frac{1}{\lambda\binom{n+\lambda}{n}}\left(\sum_{m=1}^{n}\frac{\lambda{(\lambda+1)(\lambda+2)\cdots (\lambda+m-1)}b_m}{m!}\right)=\\
   &&\frac{1}{\binom{n+\lambda}{n}}[\sum_{m=1}^{n}\binom{\lambda-1+m}{m}b_m]\\
   \end{eqnarray*}
   If $\lambda =0$ it is obvious.
 \end{proof}

\begin{lemma}\label{lemmav}
For every \( \lambda \in \mathbb{C} \setminus \{-1, -2, \dotsc, -n\} \) and for every \( n \geq 1 \), we have:
\begin{equation}
\sum_{k=1}^n \binom{n}{k} \frac{a_k}{k+\lambda}
= n! \sum_{m=1}^n \frac{b_m}{m!(\lambda + m)(\lambda + m + 1) \cdots (\lambda + n)}
- b_0 n! \sum_{m=1}^n \frac{1}{m!(\lambda + m)(\lambda + m + 1) \cdots (\lambda + n)}.
\label{suce1}
\end{equation}

The last Lemma is  a new form of Boyazhiev's theorem.
\end{lemma}

\medskip

The proof of this version of the lemma is simply based on the observation that Boyadzhiev's approach~\cite{khristo} leads to the same result under an equivalent transformation.

\medskip

The proof of the following theorem is based on Lemmas~\ref{lemma} and~\ref{lemmav}.
\begin{theorem}\label{theorem}
 Let \( \lambda \in \mathbb{C} \setminus \{-1, -2, \dotsc, -n\} \) and \( n \geq 1 \). Then, based on Lemmas~\ref{lemma} and~\ref{lemmav}, we have:
\begin{equation}
\sum_{k=1}^n \binom{n}{k}\frac{a_k}{k+\lambda}=\begin{cases}\sum_{m=1}^n\frac{b_m}{m}
-b_0 H_n.&{\text{if }\lambda=0},\\\label{suce11}\\
=\frac{1}{\lambda \binom{n+\lambda}{n}}[\sum_{m=1}^{n}\binom{\lambda-1+m}{m}b_m]
-b_0\frac{1}{\lambda\binom{n+\lambda}{n} [\binom{\lambda+n}{n}-1]}.&{\text{otherwise.}}
\end{cases}
\end{equation}

If $a_0 = 0$ (so that $b_0 = 0$ as well), then we obtain a new form of Boyazhiev's theorem.

\end{theorem}

In the following examples, we will use Lemma (\ref{lemma}) and Theorem (\ref{theorem}).

\idi{In the first case for $\lambda=0$, we have the following relationship
$$\sum_{k=1}^n\binom{n}{k}\frac{a_k}{k}=%\sum_{m=1}^n\frac{b_m}{m}-b_0\sum_{m=1}^n\frac{1}{m}=
\sum_{m=1}^n\frac{b_m}{m}-b_0H_n.$$}

\begin{itemize}
\item 
For $a_n = (\alpha - 1)^n$, the following relation between generalized harmonic numbers and harmonic numbers holds.
$$H_n(\alpha)=H_n+\sum_{k=1}^n\binom{n}{k}\frac{(\alpha-1)^k}{k}.$$
\item Let $a_n = (-2)^n$. Then the following relation between skew-harmonic numbers and harmonic numbers holds..
$$H_n(-1):=H_n^-=H_n+\sum_{k=1}^n\frac{(-2)^k}{k}.$$
\end{itemize}
\begin{itemize}
 \item Second case  $\lambda=0$, and  $a_n=1$ for $n\in \BBN.$ then we have\par
$\sum_{k=1}^n\binom{n}{k}\frac{1}{k}=\sum_{m=1}^n \frac{2^m}{m}-H_n.$ %
\item third case for $\lambda=1$, we have this relationship
$$\sum_{k=1}^n\binom{n}{k}\frac{a_k}{k+1}=\frac{1}{n+1}\left( \sum_{m=1}^nb_m-nb_0\right).$$
 \idi{ The last theorem (\ref{theorem}) is a generalized version of an important identity. It is a well-known one Knuth \cite{knuth}, page [202] and Flajolet (\cite{flajolet}).
   $$ \sum_{k=0}^n\binom{n}{k}\frac{(-1)^{k}}{k+\lambda}=\frac{1}{\lambda\binom{\lambda+n}{n}}:=\frac{n!}{\lambda(\lambda+1)(\lambda+2)\dots(\lambda+n)}.$$}

\end{itemize}
\idi{ We are going to calculate the following relationship without using the Euler's Beta function, \label{app1}
to see  Boyadzhiev's Lemma 3 in\cite{khristo2}.That is a generalization of Gould's identity for $a=1.$

\begin{equation}
 \sum_{k=j}^n\binom{n}{k}\binom{k}{j}\frac{(-1)^k(a)^{k}}{k}=\sum_{t=1}^n\frac{\binom{t}{j}(-a)^j\left(1-a\right)^{t-j}}{t}.\label{eulerbnew}
\end{equation}

\begin{proof}

 We apply   the  Theorem (\ref{theorem})  to the next identity.
 $ \sum_{k=j}^n\binom{n}{k}\binom{k}{j}{(-a)^{k}}=\binom{n}{j}(-a)^j\left(1-a\right)^{n-j},$
  
  then we have 
 \begin{equation}
 \sum_{k=j}^n\binom{n}{k}\binom{k}{j}\frac{(-1)^k(a)^{k}}{k}=\sum_{t=1}^n\frac{\binom{t}{j}(-a)^j\left(1-a\right)^{t-j}}{t}.\label{eulerbnew2}
\end{equation}

 \end{proof}}

 \idi{New question
   \begin{equation}
\dsum{k=1}{n}{(-1)^k\binom{n}{k}H_k(\alpha)k^p}=???\label{new}
\end{equation}}

Having finished the binomial transformations of the cases of certain rational functions, we move 
to  deal with generalized harmonic numbers instead of harmonic numbers with the aim of solving a problem of  Boyadzhiev and Coffey.

{\section{Pan}}

In this section, we provide an alternative proof of Pan's Theorem, splitting it into two parts. The second part is not present in the original paper, In addition, we are working in, P(I. Niven \cite{niven}),  the integral ring of the normal power series as follows.\par

 Without using the convergent series, we give a new proof of a Boyadzhiev's Lemma in the paper of  Boyadzhiev\cite{khristo3}. This proof is close in spirit to the classical article Gould (\cite{gould}) and   Boyadzhiev (\cite{khristo3}).\par
\begin{lemma}\label{lemma2pan}
Given a function defined by a formal power series
$$f(t)=\sum_{n=0}^{\infty}a_nt^n,$$
 then the following representation holds

\begin{equation}
\frac{1}{1-\lambda t}f\left(\frac{\mu}{\lambda}\left(\frac{\lambda t}{1-\lambda t}\right)\right)=\sum_{n=0}^{\infty}t^n\left( \sum_{k=0}^{n}
\binom{n}{k}{\mu}^k{\lambda}^{n-k}a_k\right).\label{panequa1}
\end{equation}
where $\lambda,\mu$ are appropriate parameters.
\end{lemma}
\begin{proof}
 \begin{eqnarray*}
  &&\sum_{n=0}^{\infty}t^n\left( \sum_{k=0}^{n}\binom{n}{k}{\mu}^k{\lambda}^{n-k}a_k\right)=\\
&&\sum_{n=0}^{\infty}(\lambda t)^n\left( \sum_{k=0}^{n}\binom{n}{k}(\frac{\mu}{\lambda})^ka_k\right)=
\sum_{k=0}^{\infty}(\frac{\mu}{\lambda})^ka_k\left( \sum_{k=n}^{\infty}
\binom{n}{k}(\lambda t)^n\right)\\
&&\sum_{k=0}^{\infty}(\frac{\mu}{\lambda})^ka_k\left( \sum_{n=0}^{\infty}
\binom{n+k}{k}(\lambda t)^{n+k}\right)=\sum_{k=0}^{\infty}(\frac{\mu}{\lambda})^ka_k(\lambda t)^k\left( \sum_{n=0}^{\infty}\binom{n+k}{k}(\lambda t)^{n}\right)\\
&&\sum_{k=0}^{\infty}(\frac{\mu}{\lambda})^ka_k(\lambda t)^k(1-\lambda t)^{-k-1}=\frac{1}{1-\lambda t}\sum_{k=0}^{\infty}(\frac{\mu}{\lambda})^ka_k\left(\frac{\lambda t}{1-\lambda t}\right)^k=\\
&&\frac{1}{1-\lambda t}f\left(\frac{\mu}{\lambda}\left(\frac{\lambda t}{1-\lambda t}\right)\right)\\
 \end{eqnarray*}
%\end{description}
\end{proof}
We provide a new proof of Pan's theorem through the necessary use of formal power series.
\begin{theorem}[Boyadzhiev and Pan]\label{teorempan}
Let $H_k(\alpha)$ be the generalized Harmonic numbers. Then, we have for \par
  $n\geq1,$
$$ \sum_{k=0}^n\binom{n}{k}\mu^k\lambda^{n-k}H_{k}(\alpha)=\begin{cases}\ (\mu+\lambda)^n\left(H_n( \frac{\lambda+\mu\alpha}{\mu+\lambda})-H_n( \frac{\lambda}{\mu+\lambda}\right)&\text{if }\mu+\lambda\neq 0.\\

\frac{\lambda^n}{n}\left((1-\alpha)^{n}-1\right)& \text{if}\;\,\mu+\lambda=0.\end{cases} $$

\end{theorem}

Here we apply the equation (\ref{panequa1}) to the function $f(t)=\frac{\ln(1-\alpha t)}{1-t}=-\sum_{n=1}^{\infty}H_n(\alpha)t^n$. This give the proof of Pan's theorem (\ref{teorempan})
\begin{proof}
  $\frac{1}{1-\lambda t}f\left(\frac{\mu t}{1-\lambda t}\right)=\frac{1}{1-\lambda t}\frac{\ln\left(1- \left(\frac{\alpha\mu t}{1-\lambda t}\right)\right)}{1-\left(\frac{\mu t}{1-\lambda t}\right)}$\par
 $ \frac{\ln\left(1- \left(\frac{\alpha\mu t}{1-\lambda t}\right)\right)}{(1-\lambda t)(1-\frac{\mu}{\lambda}\left(\frac{\lambda t}{1-\lambda t})\right)}=\frac{\ln(1-(\lambda+\alpha\mu) t-\ln(1-\lambda t)}{1-(\lambda+\col{red}{\mu)}t}=
 \left(-\sum_{n=1}^{\infty}(\frac{(\lambda+\alpha\mu)^n}{n}t^n+\sum_{n=1}^{\infty}\frac{\lambda^n}{n}t^n\right)\sum_{n=0}^{\infty}(\lambda+\mu)^nt^n=$\par
 $\left(\sum_{n=1}^{\infty}\sum_{k=1}^n\frac{\lambda^k}{k}(\lambda+\mu)^{n-k}\right)t^n
 -\left(\sum_{n=1}^{\infty}\sum_{k=1}^n\frac{(\lambda+\alpha\mu)^k}{k}(\lambda+\mu)^{n-k}\right)t^n=$\par
 $ \sum_{n=1}^{\infty}\left(\sum_{k=1}^n\frac{(\frac{\lambda}{\lambda+\mu})^k}{k}(\lambda+\mu)^nt^n\right)- \sum_{n=1}^{\infty}\left(\sum_{k=1}^n\frac{(\frac{\lambda+\alpha\mu}{\lambda+\mu})^k}{k}(\lambda+\mu)^nt^n\right)=$\par
 $$ \sum_{n=1}^{\infty}\left(H_n(\frac{\lambda}{\lambda+\mu})-H_n(\frac{\lambda+\alpha\mu}{\lambda+\mu})\right)(\lambda+\mu)^nt^n. $$
in case $\lambda+\mu=0$ we use the same method.\par 

  \end{proof}
We now provide several examples to demonstrate the importance of our theorem.
\idi{
By Theorem(\ref{teorempan}) we have
$$\sum_{k=0}^n(-1)^k\binom{n}{k}H_{k}(\alpha)=\frac{1}{n}\left((1-\alpha)^{n}-1\right).\label{idi1}$$
In the case $\alpha =1$ we have the classical identity.}
\idi{We know that  $H_k^-=-H_k(-1)$  therefore
 $$\sum_{k=0}^n\binom{n}{k}H_k^-=2^nH_n(\frac{1}{2}).$$}

\idi{
Our formula differs from Robert Frontczak's in this identity.
$$\sum_{k=0}^n\binom{n}{k}2^kH_k^-=- (3)^n\left(H_n( \frac{-1}{3})-H_n( \frac{1}{3})\right).$$}
\idi{Setting $\lambda=1=\mu$, in  the Theorem (\ref{teorempan})yields a new generalization of the known identity 
 (see, for example) Spivey(\cite{spivey}))\par 
  $\sum_{k=1}^n\binom{n}{k}{H_{k}(\alpha)}={2^{n}}\left(H_n(\frac{1+\alpha}{2})-H_n(\frac{1}{2}\right)$\par
  It's know that $\lim_{n->+\infty}\sum_{k=0}^n\frac{1}{k2^k}=2$(substitute $\frac{-1}{2}$ into the Maclaurin series for $\ln(1+x)$,and thus we have the following expression for $\ln 2:$
$$ \ln 2=\lim_{n->+\infty}\left(H_n(\frac{1+\alpha}{2})-2^{-n}\sum_{k=0}^n\binom{n}{k}H_n(\frac{1}{2})\right)$$}
\idi{The case $\alpha=1 $ is well known. \par $\sum_{k=1}^n\frac{H_{k}(\alpha)}{k}=\begin{cases}\frac{1}{2}\left(H_n^2+H_n^{(2)}\right)&\text{if}\, \alpha=1,\\
H_n(\alpha)\cdot H_n+H_n(\alpha)^{(2)}-\sum_{k=1}^n\alpha^k\frac{H_k}{k}&\text{if}\, \alpha\ne 1\\ 
  \end{cases}$\par   }
\vspace{25pt}

{The subsequent theorem functions as a generalization of two existing theorems. The first of these is a theorem by Boyadzhiev on harmonic numbers, and the second is a theorem by Fronczak on antisymmetric harmonic numbers.}

 \begin{theorem}
\label{theoremNew version}
Given the sequences $a_n = (-1)^n H_n(\alpha)$ and $c_n$, we consider their binomial transforms. In particular,  
$
b_n  = \sum_{k = 0}^n \binom{n}{k} (-1)^k H_k(\alpha)= \frac{1}{n}\big((1 - \alpha)^n - 1\big)
$
and
$
d_n = \sum_{k = 0}^n \binom{n}{k} (-1)^{\,n-k} c_k.
$

then we have the identity.
\begin{eqnarray}
 \lefteqn{\sum_{k=0}^n\binom{n}{k}(-1)^kH_k(\alpha)c_k\hspace{-2pt}=}\nonumber\\
&&{(-1)^n}d_nH_n{(\alpha)}\hspace{-2pt}-\hspace{-2pt}\sum_{m=0}^{n-1}d_m\frac{(-1)^{m}}{n-m}\hspace{-2pt}+\hspace{-2pt} \sum_{m=0}^{n-1}d_m\hspace{-2pt}\left(\sum_{t=1}^{n}\frac{\binom{t}{n-m}(-1)^m(1-a)^{n-m}\left(a\right)^{t-(n-m)}}{t}[\alpha\ne 1]\right)\\\label{eqnnew8}
\end{eqnarray}
\end{theorem}
 \begin{proof} Afther the Boyadzhiev'sTheorem(\cite{khristo2}) we have 
 \begin{equation}\sum_{k=0}^n(-1)^k\binom{n}{k}H_k(\alpha)c_k=\binom{n}{n}d_n\nabla^n b_n+\sum_{m=0}^{n-1}d_m\binom{n}{m}\nabla^m b_n.\label{eqnnew9}\end{equation}
 The first sum on the RHS in the equation(\ref{eqnnew9}) is same as .
 \begin{eqnarray}
  \lefteqn{\binom{n}{n}d_n\binom{n}{m}\nabla^n b_n=}\nonumber\\
  &&d_n\sum_{j=0}^n\binom{n}{j}\binom{j}{n-n}(-1)^{n-j}b_j=d_n\sum_{j=0}^n\binom{n}{j}(-1)^{n-j}b_j=d_n(-1)^nH_n(\alpha).\label{end1}
 \end{eqnarray}
The second sum on the RHS in the equation(\ref{eqnnew9}) is same as 
  \begin{eqnarray}
  \lefteqn{\sum_{m=0}^{n-1}d_m\binom{n}{m}\nabla^m b_n=}\nonumber\\
  &&\sum_{m=0}^{n-1}d_m\left(\sum_{j=0}^n\binom{n}{j}\binom{j}{n-m}(-1)^{n-m}b_j\right)=\nonumber\\
&&\sum_{m=0}^{n-1}d_m\left(\sum_{j=1}^n\binom{n}{j}\binom{j}{n-m}(-1)^{n-j}\left\{\frac{1}{j}((1-j)^j-1)\right\}\right)\nonumber\\
  &&=\sum_{m=0}^{n-1}d_m(-1)^n\sum_{j=1}^n\binom{n}{j}\binom{j}{n-m}(-1)^{j}\left\{\frac{1}{j}((1-j)^j)\right\}+\nonumber\\
&&\sum_{m=0}^{n-1}d_m(-1)^n\sum_{j=1}^n\binom{n}{j}\binom{j}{n-m}(-1)^{j}\left\{\frac{1}{j}(-1)\right\}\;\;\;\quad\label{eqnnew10}
   \end{eqnarray}
   The second sum on the RHS in the equation(\ref{eqnnew10}) is same as 
  \begin{eqnarray}
  \lefteqn{-\sum_{m=0}^{n-1}d_m(-1)^n\sum_{j=1}^n\binom{n}{j}\binom{j}{n-m}(-1)^{j}\frac{1}{j}=}\nonumber\\
  &&-\sum_{m=0}^{n-1}d_m(-1)^n(-1)^{n-m}\frac{1}{n-m}=-\sum_{m=0}^{n-1}d_m\frac{(-1)^{m}}{n-m}\nonumber\\
  &&-\sum_{m=0}^{n-1}d_m(-1)^n(-1)^{n-m}\frac{1}{n-m}=-\sum_{m=0}^{n-1}d_m\frac{(-1)^{m}}{n-m}\label{end2}
  \end{eqnarray}
  The first sum on the RHS in the equation(\ref{eqnnew10}) is same as 
  \begin{eqnarray}
  \lefteqn{\sum_{m=0}^{n-1}d_m(-1)^n\sum_{j=1}^n\binom{n}{j}\binom{j}{n-m}(-1)^{j}\left\{\frac{1}{j}((1-j)^j)\right\}=}\nonumber\\
 &&(-1)^n\sum_{m=0}^{n-1}d_m\sum_{j=n-m}^n\binom{n}{j}\binom{j}{n-m}(-1)^{j}\left\{\frac{1}{j}((1-j)^j)\right\}=\nonumber\\
 &&(-1)^n\sum_{m=0}^{n-1}d_m\sum_{t=1}^n\frac{\binom{t}{n-m}(-(1-\alpha))^{n-m}\alpha^{t-(n-m)}}{t}=\nonumber\\
 &&\sum_{m=0}^{n-1}d_m(-1)^m\sum_{t=1}^n\frac{\binom{t}{n-m}(1-\alpha))^{n-m}\alpha^{t-n+m}}{t}.\label{end3}
  \end{eqnarray}

We can prove the theorem after  equations (\ref{end1},\ref{end2} and \ref{end3}).
 \end{proof}
 \par
 
 \begin{example}\label{examplersp}
  We now give  several examples to demonstrate the importance of our theorem(\ref{theoremNew version})\par
 We present a list of entities to which equation  (\ref{eqnnew8}) can be applied to create a new identity involving generalized harmonic numbers.
\begin{itemize}
\item $n^{\beta}=\sum_{k=0}^n\binom{n}{k}k!\st{\beta}{k}, \quad c_n=n^{\beta},\quad d_n=n!\st{\beta}{n}.(Stirling\; numbers\; of\; second\; kinds\;\beta\in\BBC).$
\item $\frac{(-1)^{n-1}}{n}=\sum_{k=0}^n\binom{n}{k}(-1)^{k}H_k ,\quad c_n=(-1)^{n-1}H_n,\quad d_n=\frac{(-1)^{n-1}}{n}(Harmonic numbers). $
\item $F_n=\sum_{k=0}^n\binom{n}{k}(-1)^{k-1}F_k, \quad c_n=F_n,\quad d_n=(-1)^{n-1}F_n(Fibonacci \;numbers).$
\item $F_{2n}=\sum_{k=0}^n\binom{n}{k}F_k, \quad c_n=F_{2n},\quad d_n=F_n(Fibonacci \;numbers). $
\item $L_n=\sum_{k=0}^n\binom{n}{k}(-1)^{k}L_k, \quad c_n=L_n,\quad d_n=(-1)^{n}L_n(Lucas \;numbers).$
\item $L_{2n}=\sum_{k=0}^n\binom{n}{k}L_k, \quad c_n=L_{2n},\quad d_n=L_n(Lucas \;numbers). $
\item $(-1)^nB_{n}=\sum_{k=0}^n\binom{n}{k}B_k, \quad c_n=(-1)^nB_{n},\quad d_n=B_n.(Bernoulli \;numbers). $
\item $L_{n}(x)=\sum_{k=0}^n\binom{n}{k}\frac{(-x)^k}{k!}, \quad c_n=L_{n}(x),\quad d_n=\frac{(-x)^n}{n!} (Laguerre\, polynomials).$
\end{itemize}
\end{example}

\par
The next section  we generalize a special case of  Boyadzhiev, Pan's formula with the aim of generalize  the next question $S_n^p(z):=\sum_{j=0}^n\binom{n}{j}j^p H_jz^j.$ Then we show some examples at the end we also  prove Coffey'question   (\cite{coffey})   

Our method uses  condition \col{red}{ $p\leq n$} however using Coffey's method we could remove this restriction.\par 

 {\section{ Boyadzhiev-end}}
 
 { Here, we generalize a theorem of Boyadzhiev (\cite{khristo3}) by replacing the classical harmonic numbers with generalized harmonic numbers.} We provide a closed formula for the general case. We list several theorems and set this definition:
 \begin{equation}AS_n^p(z):=\sum_{j=0}^n\binom{n}{j}j^p H_j(\alpha)z^j.\label{boyacoff}\end{equation}
 \begin{description}
 \item First case we take $z\ne -1.$
 \begin{itemize}
\item By theorem (\ref{teorempan})   we have  this relation\par  $\sum_{k=0}^n\binom{n}{k}z^kH_k(\alpha)=(1+z)^n\left(H_n(\frac{1+\alpha z}{1+z})-H_n(\frac{1}{1+z})\right)$
 \item By Khristo and Sanchez, we have this relation  \par  
 $\sum_{k=0}^n\binom{n}{k}k^pa_k=\sum_{k=0}^n\sum_{l=0}^p\sum_{j=0}^p(-1)^l\binom{n-l}{k}\binom{n}{l}\binom{n-1}{j-1}j!\st{p}{j}a_k=$\par
$\sum_{l=0}^p\sum_{j=0}^p(-1)^l\binom{n}{l}\binom{n-1}{j-1}j!\st{p}{j}b_{n-l}$ where  $b_n=\sum_{k=0}^n\binom{n}{k} a_k.$ See Definition 1.

\begin{eqnarray}
 \lefteqn{ AS_n^p(z):=\sum_{j=0}^n\binom{n}{j}j^p H_j(\alpha)z^j=}\\
&&\hspace{-50pt} \sum_{l=0}^p\sum_{j=0}^p(-1)^l\binom{n}{l}\binom{n-l}{j-l}j!\st{p}{j}(1+z)^{n-l}\left(H_{n-l}(\frac{1+\alpha z}{1+z})-H_{n-l}(\frac{1}{1+z})\right)\label{newcoffey}\\
 \end{eqnarray}
\end{itemize}
\item  Second  case we take $z=-1.$
\begin{itemize}
\item By Identity (\ref{idi1})    we have  this relation\par 
 $\sum_{k=0}^n\binom{n}{k}(-1)^kH_k(\alpha)=\frac{1}{n}\left((1-\alpha)^n)-1\right)$
 \item By Khristo and Sanchez, we have this relation %\col{red}{CONTROLLARE p,n??} 
 \par  
 $\sum_{k=0}^n\binom{n}{k}k^pa_k=\sum_{k=0}^n\sum_{l=0}^p\sum_{j=0}^p(-1)^l\binom{n-l}{k}\binom{n}{l}\binom{n-1}{j-1}j!\st{p}{j}a_k=$\par
$\sum_{l=0}^p\sum_{j=0}^p(-1)^l\binom{n}{l}\binom{n-1}{j-1}j!\st{p}{j}b_{n-l}$ where  $b_n=\sum_{k=0}^n\binom{n}{k} a_k.$ See Definition 1.
\begin{eqnarray}
 \lefteqn{ AS_n^p(z):=\sum_{j=0}^n\binom{n}{j}j^p H_j(\alpha)(-1)^j=}\\
&&\hspace{-50pt} \sum_{l=0}^p\sum_{j=0}^p(-1)^l\binom{n}{l}\binom{n-l}{j-l}j!\st{p}{j}(\frac{1}{n-l}\left((1-\alpha)^{n-l}-1\right)\label{newcoff}
 \end{eqnarray}
 \item  Third   case we take $z=-1, \alpha=1.$
 \begin{eqnarray}
  AS_n^p(z):=\sum_{j=0}^n\binom{n}{j}j^p H_j(1)(-1)^j=(-1)^nn!\st{p}{n}H_n-\dsum{\col{red}{k=0}}{n-1}{\frac{(-1)^kk!{\binom{n}{k}}}{n-k}}\label{newcoffey1}
\end{eqnarray}
We give the proof in this case 
 Let $ p<n,$ then the LHS in(\ref{newcoffey1}) is  \par
  \begin{eqnarray*}
  \lefteqn{\dsum{k=0}{n}{(-1)^k\binom{n}{k}H_kk^p}=}\\
&&\sum_{k=0}^n\sum_{l=0}^{p}\sum_{j=0}^p(-1)^l\binom{n-l}{k}\binom{n}{l}\binom{n-l}{j-l}j!\st{p}{j}(-1)^kH_k=\\
&&\sum_{l=0}^{p}\sum_{j=0}^p(-1)^l\frac{-1}{n-l}\binom{n}{l}\binom{n-l}{j-l}j!\st{p}{j}\srel{*}{=}\sum_{\col{red}{t=1}}^{n}\sum_{j=0}^p(-1)^{n-t}\frac{-1}{t}\binom{n}{t}\binom{t}{n-j}j!\st{p}{j}=\\
&&-\sum_{j=0}^{\col{red}{n-1}}\frac{(-1)^j}{n-j}j!\st{p}{j}.\\
\end{eqnarray*}

 Because $n-1\geq l\geq j>p$.  $*$ is true. The same method is used when $p\ge n$.\par
  The other method, which is very simple, involves using the theorem (\ref{theoremNew version}) and the first example of (\ref{examplersp}).
\end{itemize}
\end{description}

 To demonstrate equation (\ref{newcoffey}) in action, we will provide several examples\par 
 
\begin{example}
 \begin{description}

  \item [If $p=0$ ]  then we have
  $$AS_n(z)=\sum_{k=0}^n\binom{n}{k}(-1)^kz^kH_k(\alpha)=(1+z)^n\left(H_n(\frac{1+\alpha z}{1+z})-H_n(\frac{1}{1+z})\right)$$

 \item [If  $p=1$ ] then we have $\binom{n}{k}k= n\binom{n}{k}-n\binom{n-1}{k}$\;\par
  $AS_n^1(z)=n(1+z)^{n}\left(H_{n}(\frac{1+\alpha z}{1+z})-H_{n}(\frac{1}{1+z})\right)-n(1+z)^{n-1} \left(H_{n-1}(\frac{1+\alpha z}{1+z})-H_{n-1}(\frac{1}{1+z}) \right)=$\par
 \begin{eqnarray}
   n(1\col{red}{+}z)^{n-1}\left\{zH_{n-1}(\frac{1+\alpha z}{1+z})-zH_{n-1}(\frac{1}{1+z})+\frac{1+z}{n}\left[\left(\frac{1+\alpha z}{1+z}\right)^n-\left(\frac{1}{1+z}\right)^n\right]\right\}\label{exemple1}
   \end{eqnarray}
  \item[If  $p=2$ ] then we have $ n^2\binom{n}{k}-n(2n-1)\binom{n-1}{k} +n(n-1)\binom{n-2}{k}$\;\par

  \begin{eqnarray*}
 \lefteqn{AS_n^2(z)=}\\
 && \hspace{-50pt} n^2(1+z)^{n}\left(H_{n}(\frac{1+\alpha z}{1+z})-H_{n}(\frac{1}{1+z})\right))-\\
 &&n(2n-1)(1+z)^{n-1}\left(H_{n-1}(\frac{1+\alpha z}{1+z})-H_{n-1}(\frac{1}{1+z}\right)+\\
 &&  n(n-1)(1+z)^{n-2}\left(H_{n-2}(\frac{1+\alpha z}{1+z})-H_{n-2}(\frac{1}{1+z})\right)=\\
 &&n(1+z)^{n-2}\left\{\left(H_n(\frac{1+\alpha z}{1+z})-H_n(\frac{1}{1+z})\right)\left(n(1+z)^2-(2n-1)(1+z)+(n-1)\right)\right.+\\
 &&\left .\frac{1+\alpha z}{1+z}^{n-1}(1+z)\left(\frac{1+\alpha z}{1+z}-1)(1+z)+\frac{2nz-z}{n-1}\right)-\frac{nz}{(n-1)(1+z)^{n-2}}\right\}
 \end{eqnarray*}
\item [If $p=3$] As observed in (\ref{sanchez}), we have
   $\binom{n}{k}k^3= n^3\binom{n}{k}-n(3n^2-3n+1)\binom{n}{k} +3n(n-1)^2\binom{n-2}{k}$\par$n(n-1)(n-2)\binom{n-3}{k}.$\par 
   \begin{eqnarray*}
 \lefteqn{AS_n^3(z)=}\\
 && \hspace{-50pt} n^3(1+z)^{n}\left(H_{n}(\frac{1+\alpha z}{1+z})-H_{n}(\frac{1}{1+z})\right))-\\
 &&n(3n^2-2n+1)(1+z)^{n-1}\left(H_{n-1}(\frac{1+\alpha z}{1+z})-H_{n-1}(\frac{1}{1+z}\right))+\\
  && 3 n(n-1)^2(1+z)^{n-2}\left(H_{n-2}(\frac{1+\alpha z}{1+z})-H_{n-2}(\frac{1}{1+z})\right)=\\
 &&  n(n-1)(n-2)(1+z)^{n-3}\left(H_{n-2}(\frac{1+\alpha z}{1+z})-H_{n-2}(\frac{1}{1+z})\right)=.\\
 &&...........\\
 &&nx^{n-3}\left\{( H_{n-3}(\frac{1+\alpha z}{1+z})- H_{n-3}(\frac{1}{1+z}))\left(n^2(1+z)^3-(3n^2-3n+1)(1+z)^2\right)+ \right.\\
 &&\left. 3(n-1)^2(1+z)-(n-1)(n-2) -3(1-n)^3 +(1-n)(1-3n+3n^2)(1+z)\right.\\
 &&\left.-n^2(1-n)(1+z)^2+c(2-n)(1-3n+3n^2)(1+z)+\frac{1+\alpha z}{1+z}n^2(-2+n)(1+z)^2\right.\\
 &&\left.+\frac{1+\alpha z}{1+z}^2n(1-n)(2-n)(1+z)^2+\right.\\
 &&\left.(-1+(1+z))(1-4n-n^2(-2+z)+n^3(z))). \right\}\\
 \end{eqnarray*}

 \end{description}
 \end{example}
 
 The equation (\ref{boyacoff}) in the particular case $\alpha =1$ was studied by Coffey \cite{coffey} using hypergeometric functions, while Boyadzhiev \cite{khristo3} proposed an approach based on analytic functions. In this paper, we present a completely different method that avoids the use of both hypergeometric and analytic functions altogether.

  {\section{Conclusion}}
  
  \centerline{\Large Open problems.}

 \centerline{ Definition and Identity}
 \vspace{30pt}
  \begin{description}

 \item[1).] We have the following expansion of  Boyadzhiev (\cite{khristo3}).
  \begin{equation}
  \frac{\log(1-\alpha t)}{1-t}=-\sum_{n=1}^{\infty}H_n(\alpha)t^n \label{def}
  \end{equation}
 \item[1.1).] With $\alpha =1$ in the  identity (\ref{def}) we obtain the generating function of the harmonic numbers
 $$\frac{\log(1- t)}{1-t}=-\sum_{n=1}^{\infty}H_nt^n. $$
 \item[1.2.] 
 With $\alpha =-1$ in the  identity (\ref{def}) we obtain the generating function of the skew-harmonic numbers Sofo (\cite{sofo}).
 $$\frac{\log(1+ t)}{1-t}=-\sum_{n=1}^{\infty}H_n^{-1}t^n. $$
  \item[2).]
 $\sum_{k=0}^n(-1)^k\binom{n}{k}H_{k}(\alpha)=\frac{1}{n}\left((1-\alpha)^{n}-1\right).$
\item[3).] $\sum_{k=1}^n\frac{H_{k}(\alpha)}{k}\srel{?}{=}\begin{cases}\frac{1}{2}\left(H_n^2+H_n^{(2)}\right)&\text{if}\, \alpha=1,\\
H_n(\alpha)\cdot H_n+H_n(\alpha)^{(2)}-\sum_{k=1}^n\alpha^k\frac{H_k}{k}&\text{if}\, \alpha\ne 1\\ 
  \end{cases}$

 \item[4).]  By Boyadzhiev and by (\ref{idi1}) we have $\sum_{k=1}^n(-1)^k\frac{H_{k}(\alpha)}{k}=H_n(1-\alpha)^{(2)}-H_n(1)^{(2)}$.\par
 We find the relationship we see in the literature.\par
  $\sum_{k=1}^n(-1)^k\frac{H_{k}(\alpha)}{k}=\begin{cases}H_n(1)^{(2)}=\sum_{k=1}^n\frac{1}{k^2}&\text{if}\, \alpha=1,\\
H_n(2)^{(2)}- H_n(1)^{(2)}=\sum_{k=1}^n\frac{2^k}{k^2}-\sum_{k=1}^n\frac{1}{k^2}&\text{if}\, \alpha=- 1.\\ 
  \end{cases}$
  \item[5).] 
  To see Example(\ref{examplersp}). here we present a list of entities to which the Theorem (\ref{theoremNew version}) can be applied to create a new identity involving generalized harmonic numbers,The following are generalizations of known identities: Boyadzhiev(\cite{khristo2})

  \end{description}

After the definitions and a few examples, we can propose the following problems: take some identity involving harmonic numbers \( H_n(1) \) or skew-harmonic numbers \( H_n(-1) \), substitute with generalized harmonic numbers \( H_n(\alpha) \), and find a new identity.

 {\bf Acknowledgment.}
 
 The author would like to express gratitude to Professor 
 Pan for his suggestions.

.

\end{document}